\documentclass[a4,twocolumn]{article}
\usepackage{import}
\usepackage{colortbl}
\usepackage{xcolor}
\usepackage[caption=false,font=footnotesize]{subfig}
\usepackage{amsmath,amssymb,amsthm}
\usepackage{graphicx}
\usepackage{multicol}
\usepackage{xfrac}

\newtheorem{prop}{Proposition}
\newtheorem{cor}[prop]{Corollary}

\newtheorem{lem}[prop]{Lemma}
\newtheorem{defn}[prop]{Definition}
\newtheorem{rem}[prop]{Remark}

\newcommand{\half}{\frac{1}{2}}
\newtheorem{example}{Example}

\newcommand{\N}{{\mathbb N}}

\renewcommand{\P}{{\mathcal P}}
\newcommand{\R}{{\mathbb R}}

\newcommand{\activeSet}{{\mathcal A}}
\newcommand{\inactiveSet}{{\mathcal I}}

\newcommand{\lb}[1]{\underline{#1}}

\newcommand{\bigG}{\underline{G}}

\newcommand{\bigw}{\underline{w}}
\newcommand{\bigE}{\underline{E}}
\newcommand{\bigB}{\underline{B}}
\newcommand{\bigA}{\underline{A}}
\newcommand{\bigx}{\underline{x}}
\newcommand{\bigu}{\underline{u}}

\newcommand{\bigH}{\underline{H}}
\newcommand{\bigF}{\underline{F}}
\newcommand{\bigY}{\underline{Y}}

\newcommand{\TODO}[1]{\textcolor{red}{\\ (TODO: #1)}} 
\newcommand{\REMARK}[1]{} 
\newcommand{\blue}[1]{\textcolor{blue}{ {#1} }}
\renewcommand{\blue}[1]{{#1}}
\newcommand{\red}[1]{\textcolor{red}{ {#1} }}
\renewcommand{\red}[1]{{#1}}
\newcommand{\green}[1]{\textcolor{violet}{ {#1} }}
\renewcommand{\green}[1]{{#1}}
\renewcommand{\a}{\mathcal{\alpha}}
\newcommand{\A}{\mathcal{A}}
\newcommand{\I}{\mathcal{I}}
\newcommand{\T}{\mathcal{T}}
\newcommand{\X}{\mathcal{X}}
\newcommand{\U}{\mathcal{U}}
\newcommand{\F}{\mathcal{F}}
\renewcommand{\i}{\iota}

\begin{document} 
	\title{On the structure of the set of active sets in constrained linear quadratic regulation}
	
	\author{Martin M\"onnigmann, Automatic Control and Systems Theory \\ 
	Department of Mechanical Engineering,
	Ruhr-Universit\"at Bochum, Germany \\
	{\tt martin.moennigmann@rub.de}}
	\maketitle
	\begin{abstract}
	The constrained linear quadratic regulation problem is solved by a continuous piecewise affine function on a set of state space polytopes. 
	It is an obvious question whether this solution can be built up iteratively by increasing the horizon, 
	i.e., by extending the classical backward dynamic programming solution for the unconstrained case to the constrained case.  
    Unfortunately, however, the piecewise affine solution for horizon $N$ is in general not contained in the piecewise affine law for horizon $N+1$.
    We show that backward dynamic programming does, in contrast, result in a useful structure for the set of the active sets that defines the solution. 
    %
    %
    Essentially, every active set for the problem with horizon $N+1$ results from extending an active set for horizon $N$, if the constraints are ordered stage by stage.   
    Consequently, the set for horizon $N+1$ can be found by only considering the constraints of the additional stage. 
    Furthermore, it is easy to detect which polytopes and affine pieces are invariant to increasing the horizon, 
    and therefore persist in the limit $N\rightarrow\infty$. 
    Several other aspects of the structure of the set of active sets become evident if the active sets are represented by bit tuples. 
    There exists, for example, a subset of special active sets that generates a positive invariant and persistent (i.e., horizon invariant) set around the origin. 
    It is very simple to identify these special active sets, and the positive invariant and persistent region can be found without solving optimal control or auxiliary optimization problems.
    The paper briefly discusses the use of these results in model predictive control. Some opportunities for uses in computational methods are also briefly summarized. 
%
%
%
%
	\end{abstract}

\section{Problem statement and introduction}
\label{sec:ProblemStatement}
We consider the constrained linear quadratic optimal control problem with finite and infinite horizons. 
The problem for finite horizon $N$ reads
$V^\star(x(0), [0, N]):=$
\begin{subequations}\label{eq:OCP}
\begin{align}\label{eq:OCPCost}
  \min\limits_{\substack{u(k),\, k= 0,\dots, N-1\\ x(k),\,k= 1, \dots, N}}
  & \half\|x(N)\|_P^2+ \half\sum\limits_{k= 0}^{N-1} \left(\|x(k)\|^2_Q+ \|u(k)\|^2_R\right)
\end{align}
subject to
\begin{align}
  x(k+1)= Ax(k)+ Bu(k),\, k= 0, \dots, N-1 
  \label{eq:OCPDynamics}
  \\
  u(k)\in\mathcal{U},\, k= 0, \dots, N-1
  \label{eq:OCPInputConstraints}
  \\
  x(k)\in\mathcal{X},\, k= 0, \dots, N-1 
  \label{eq:OCPStateConstraints}
  \\
  x(N)\in\mathcal{T},
  \label{eq:OCPTerminalConstraints}
\end{align}
\end{subequations}
where $x(0)$ is the given initial condition,
$x(k)\in\R^n$ and $u(k)\in\R^m$ are the state and input variables, respectively, and the matrices have the obvious dimensions. 
We assume $(A, B)$ to be controllable, $Q\succeq 0$, $R\succ 0$ and 
$\mathcal{X}$, $\mathcal{U}$ to be compact convex polytopes that contain the origin in their interiors. 
Furthermore,  
$P$ is assumed to be the solution to the discrete-time algebraic Riccati equation,
\REMARK{\begin{equation}\label{eq:DARE}
  P= A^\top\left(P-PB\left(R+B^\top P B\right)^{-1} B^\top P\right) A +Q
\end{equation}} 
and $\T\subset\R^n$ is assumed to be the largest set such that the solution to~\eqref{eq:OCP} 
and the solution to the unconstrained infinite-horizon problem are equal for all $x(0)\in\T$. 

The infinite-horizon problem results in the limit $N\rightarrow\infty$ if the first term in~\eqref{eq:OCPCost} and~\eqref{eq:OCPTerminalConstraints} are omitted. 
The unconstrained infinite-horizon problem results if~\eqref{eq:OCPInputConstraints} and~\eqref{eq:OCPStateConstraints} are also omitted. 
\REMARK{In the infinite-horizon cases, we need to replace min by inf.}

We briefly recall that $u= K_\infty x$ with
\begin{equation*}
  K_\infty= -\left(B^\top P B+ R\right)^{-1}B^\top P A
\end{equation*}
is the state feedback that solves the unconstrained infinite-horizon problem,
 (see, e.g., \cite[chapter 4]{Bertsekas2005}). 
Furthermore, we recall that 
\begin{equation}\label{eq:TerminalSet}
	  \T= \{\xi\in\X | (A+ BK_\infty)^k \xi\in\X_{\U}, k\ge 0\}
\end{equation}
where $\X_{\U}= \left\{\xi\in\X | K_\infty \xi\in\U\right\}$~\cite{Sznaier1987}. Some properties of $\T$ are summarized in the notation section.

Let $\bigu_N^\star(x(0))$  refer to the vector in $\R^{mN}$ that results from stacking the optimal input sequence
\begin{equation}\label{eq:FiniteOptimalInputSequence}
  u^\star(0),\, u^\star(1),\dots,\, u^\star(N-1)
\end{equation}
for \eqref{eq:OCP}. 
Let $\F_N$ refer to the set of initial states $x(0)$ for which~\eqref{eq:OCP} has a solution, where $\F_N\ne\emptyset$ 
since $\F_N\supseteq\dots\supseteq\F_1\supseteq\T\ne\emptyset$.
\REMARK{
  Show that the resulting closed-loop system inherits the stability properties that pertain on $\mathcal{T}$. In short, show stability.
}

The paper addresses the following problem:
It is known that $\bigu_N^\star:\mathcal{F}_N\rightarrow\R^{mN}$ is a continuous piecewise affine function on a partition of $\mathcal{F}_N$ into a finite number of polytopes $\P_{N,1}$, $\P_{N,2}, \dots$~\cite{Bemporad2002-Automatica} (see also~\cite{Malanowski1987}). 
It is an obvious question whether this piecewise structure can be built up iteratively 
(i.e., starting from $\bigu^\star_0:\mathcal{F}_0\rightarrow\R^{mN}$ with $\mathcal{F}_0= \T$ and finding $\F_1, \F_2, \dots$) 
or recursively (i.e., starting from some $\F_N$ and investigating $\F_{N-1}, \F_{N-2}, \dots$). 
Unfortunately, the piecewise affine and polytopic geometry of $\bigu_N^\star$ and $\bigu_{N+1}^\star$ are not related in any obvious way. 
For example, a polytope for horizon $N$ may or may not be a polytope for horizon $N+1$ (see Figure~\ref{fig:PolytopicPartitions}). 
It is the purpose of the paper to explain that the sought-after structure does indeed exist for the active sets of~\eqref{eq:OCP}, and to relate this algebraic  structure of the set of active sets to the geometric structure of polytopes and the affine functions defined on them. 
\begin{figure}
  \centering
  \scriptsize{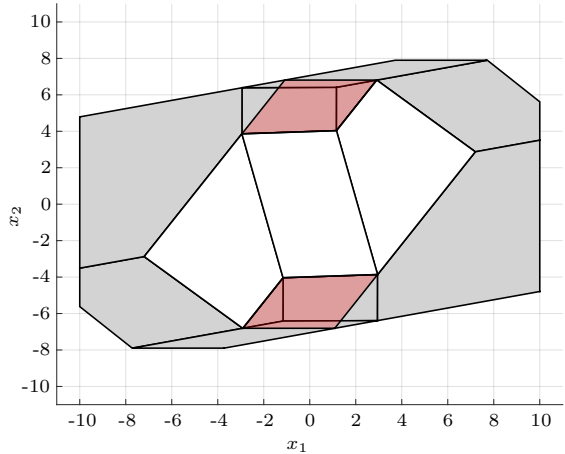}
  \caption{Partitions of the solutions to~\eqref{eq:OCP} for Example~\ref{ex:ActiveSetExtension} and $N=1, 2$. White polytopes exist for both $N=1$ and $N=2$, red polytopes exist for $N=1$ only, grey polytopes exist for $N=2$ only.}
  \label{fig:PolytopicPartitions}
\end{figure} 

Dynamic programming and the principle of optimality, which are instrumental in Section~\ref{sec:ComputationalAspects}, are fundamental techniques and have been standard tools for the derivation of the solution to the unconstrained infinite-horizon problem for decades (see, e.g., \cite{AndersonMoore1971}).
Their application to the constrained problem is hampered by the piecewise quadratic structure of the optimal cost function. 
Several publications have addressed this problem with a focus on model predictive control (MPC). 
Mu{\~n}oz de la Pe{\~n}a et al.~\cite{Munoz2004} present an algorithm for the explicit construction of the piecewise affine MPC law 
by backward dynamic programming
based on techniques proposed in~\cite{Mayne2002,Seron2002}. 
Bakara{\v{c}} et al.~\cite{Bakarac2018} show how to approximate the piecewise optimal cost function by a single quadratic function, which results in a considerable simplification of backward dynamic programming for the constrained case.
The inherent complexity of the problem is also evident from the number of candidate active sets (the powerset of $\{1, \dots, q\}$ if $q$ constraints exist) that need to be analyzed to find the active sets that actually define the optimal solution. 
Gupta et al.~\cite{Gupta2011} show that many candidate active sets can be disregarded if the powerset is organized as a tree and subtrees are pruned whenever they stem from a candidate set that is not an active set (see also \cite{Feller2013,Herceg2015,Oberdieck2017}). 

The results presented here are also based on an analysis of the set of active sets, but we focus on the structure of the solution of the optimal control problem~\eqref{eq:OCP}, specifically on the structure of the set of the active sets. 
The statements in Section~\ref{subsec:StagewiseConstruction} apply to a larger class of problems than introduced above (a class that is often analyzed in MPC after pragmatically relaxing requirements for stability). 
The properties summarized so far are then exploited in Section~\ref{sec:Persistency} (see the remark at the end of Section~\ref{subsec:StagewiseConstruction}). Section~\ref{sec:ComputationalAspects}  briefly summarizes some computational aspects. Some conclusions and opportunities for future work are stated in Section~\ref{sec:Conclusions}. 

\REMARK{Exploit the order of the constraints in a way reminiscent of the structure exploitation proposed in Wang and Boyd.}

\REMARK{Consider the minimum time problem like this: Note that we consider the problem to optimally steer the system to $\mathcal{T}$ in a finite number of at most $N$ steps and assume a known controller (if any such controller is necessary) can take over on $\mathcal{T}$. Since the contribution for $N\rightarrow\infty$ on $\mathcal{T}$ is not taken into account, it is actually more consistent not to account for the cost function contributions after $\mathcal{T}$ has been entered. It is therefore meaningful to consider the minimum time problem.
  \begin{equation}\label{eq:MinimumTimeProblem}
    \begin{split}
      \mbox{For any $x\in\mathcal{X}_f$, find the smallest $N_{\min} \le N$ }\\
      \mbox{such that~\eqref{eq:OCP} is feasible.}   
    \end{split}
  \end{equation}
which is, however, treated in a separate document.}
\REMARK{Show (once and for all) which cases may arise when LICQ fails. Does failure of LICQ imply a non fulldim polytope, for example? (What is the claim in Florian2016, p.2, top, right column? What is the claim in Gupta et al.?)
Use MPC (rather than artificial QP) examples and point this aspect out in the summary?}
\REMARK{
  Check whether it is meaningful to consider the constraints to be a master set of possible inequalities from which those inequalities are selected that define the current region of interest.
}
\REMARK{Show an optimal solution to~\eqref{eq:OCP} is not in general time-optimal on $\mathcal{F}_k$ for $k< N$.}

\subsection*{Notation and preliminaries}
The terminal set $\T$ introduced in~\eqref{eq:TerminalSet} can be  computed with the procedure proposed in~\cite{Gilbert1991}. 
It does not depend on $N$ and $\T\subseteq\X$ by definition~\cite{Chmielewski1996}.

Problem~\eqref{eq:OCP} is finite-dimensional and thus solved by a finite input sequence~\eqref{eq:FiniteOptimalInputSequence} 
and a corresponding finite state sequence that results with~\eqref{eq:OCPDynamics}. 
For comparisons to the infinite-horizon problems, these sequences are extended to
\begin{equation}\label{eq:InfiniteOptimalSequences}
  u^\star(k),\, x^\star(k) \text{ for all } k\ge 0
\end{equation}
by setting $u^\star(l)= K_\infty x^\star(l)$, $x^\star(l+1)= (A+BK_\infty) x^\star(l)$, $l\ge N$. 
We collect three basic statements about the optimal input sequence~\eqref{eq:FiniteOptimalInputSequence} for later use:
	(i)~The sequence $u^\star(1), \dots, u^\star(N-1)$ is optimal for~\eqref{eq:OCP} with horizon $N-1$ and initial condition $x^\star(1)$. 
	More generally, $u^\star(l), \dots, u^\star(N-1)$ is optimal for~\eqref{eq:OCP} with horizon $N-l$ and initial condition $x^\star(l)$, where $l\in\{0, N-1\}$ is arbitrary. 
	(ii)~In general \eqref{eq:FiniteOptimalInputSequence} is not equal to the first elements in the optimal input for~\eqref{eq:OCP} with horizon $N$ and initial condition $x^\star(1)$. 
	(iii)~In general, \eqref{eq:FiniteOptimalInputSequence} is not equal to the first elements in the optimal input sequence for~\eqref{eq:OCP} with horizon $N+1$ and initial condition $x(0)$. 
\REMARK{Remark (ii) above is never used in the paper.}
\REMARK{Double check which properties of $\T$ are needed. We may need  $0\in\text{int}\,\T$, and p.i. under $(A+BK_\infty)$. Ideally, $\T$ is a compact convex polytope that contains the origin in its interior under the stated assumptions. 
Sznaier1987 introduce $\T$ in the same fashion as Chmielewski, introduce $\F_\infty$, claim but do not prove equality of the constrained and uncontrained problem on $\T$, claim but do not prove the existence of an open ball around the origin in $\T$. }

	We need to state (iii) more precisely and in a more technical fashion for later use. According to 
	Lemma 2.2 in \cite{Chmielewski1996}\footnote{The condition $x(N)\in\text{int}\,\T$ cannot be replaced by $x(N)\in\T$ as in~\cite{Chmielewski1996}. 
	See Example~\ref{ex:PersistentActiveSets}.}
	\begin{equation}\label{eq:Lemma2.2Chmielewski1996} 
	  x^\star(N)\in\text{int}\,\mathcal{T} \text{ implies } \eqref{eq:InfiniteOptimalSequences} \text{ are equal for $N$ and $N+1$,}
	\end{equation}
	and thus for all $N+l$, $l\ge 0$ and in the limit $N\rightarrow\infty$,
	but
	\begin{equation*}
	  \text{\eqref{eq:Lemma2.2Chmielewski1996} does in general not hold for $x^\star(N)\in\partial \mathcal{T}$.}
	\end{equation*}
	It may seem pedantic to consider the boundary $\partial T$ separately, since $\partial \mathcal{T}$ is a subset of $\mathcal{T}$ and $\mathcal{X}$ with measure zero. 
	There often exist, however, full dimensional polytopes $\mathcal{P}_{N,i}$ in the piecewise affine solution such that $x^\star(N)\in\partial T$ for all $x(0)\in\mathcal{P}_{N,i}$ (see Example~\ref{ex:PersistentActiveSets} and Figure~\ref{fig:PersistentActiveSets}).
	Disregarding $\partial T$ may therefore lead to full-dimensional holes in the affine solution.

\REMARK{(Check for missing assumptions: The proof of lemma~\ref{lem:PersistentActiveSets} requires $(A+BK_\infty)x\in\text{int}\,\T$ and $K_\infty x\in\text{int}\,\U$ for all $x\in\T$.}

\REMARK{Symbols with underscores such as $\bigY$ are used for $N$ here and later $N+1$.}
For any $N$, there exist $\bigH$, $\bigY$, $\bigF$, $\bigG$, $\bigE$, $\bigw$ such that problem~\eqref{eq:OCP} can equivalently be stated as the quadratic program
\begin{subequations}\label{eq:parametricQP}
  \begin{align}
    & \min_{\bigu} \half x^\top(0)\bigY x(0)+ \half\bigu^\top \bigH\bigu+ x(0)^\top\bigF\bigu
  \\
  \label{eq:parametricQPconstraints}
  & \text{subject to } \bigG \bigu\le \bigw+ \bigE x(0)
  \end{align}
\end{subequations}
after substituting~\eqref{eq:OCPDynamics},  
where $\bigu= (u^\top(0), \dots, u^\top(N-1))^\top$, and where the state sequence that results in~\eqref{eq:OCP} can be determined with~\eqref{eq:OCPDynamics}. $H$ is positive definite under the assumptions stated for~\eqref{eq:OCP}~\cite{Bemporad2002-Automatica}. 
Consequently, \eqref{eq:parametricQP} has a unique solution for all $x(0)\in\mathcal{F}_{N}$, which we denote $\bigu_N^\star(x(0))$ in accordance with $\bigu_N^\star:\mathcal{F}_N\rightarrow\R^{mN}$ introduced above.
Let $q$ refer to the number of inequality constraints in~\eqref{eq:OCP} and~\eqref{eq:parametricQP}. 
Let $q_{\U}$, $q_{\X}$ and $q_{\T}$ refer to the number of halfspaces (i.e., inequalities) required to define
$\U$, $\X$ and $\T$, respectively. 
Polytopes are understood to be the intersection of a finite number of halfspaces and bounded.  
\REMARK{Introduce a statement on redundant halfspaces.}

A constraint $i$ is called active (resp.\ inactive) for an $x(0)\in\mathcal{F}_N$
if $\bigG_i \bigu_N^\star(x(0))= \bigw_i+ \bigE_i x(0)$ (resp. $\bigG_i\bigu_N^\star(x(0))< w_i+ E_i x(0)$), where $G_i$, $w_i$, $E_i$ etc.\ refer to the $i$-th row of the respective matrix or vector. 
A constraint $i$ is called weakly active if it is active and its multiplier $\sigma_i$ introduced in~\eqref{eq:KKTparametricQP} below is zero.
$G_\A$, $G_\I$ etc.\ refer to the submatrix of $G$ with rows indicated in $\A$ and $\I$, respectively. 
For a given $x(0)\in\mathcal{F}_N$ , let $\activeSet$ and $\inactiveSet$ (or  $\mathcal{A}(x(0))$ and $\inactiveSet(x(0))$ where needed) refer to the set of active respectively inactive constraints. 
We say $\activeSet$ ($\inactiveSet$) is an active (inactive) set for~\eqref{eq:parametricQP}
and later~\eqref{eq:OptProblemBackwardInductionStep} if there exists an $x(0)\in\mathcal{F}_N$ with this set of active (inactive) constraints. 

We assume the reader to be familiar with the Karush-Kuhn-Tucker (KKT) conditions for problems with inequality and mixed inequality and equality constraints (see, e.g., \cite[Section 4.2.13]{Bazaraa2006}.
If the active set $\activeSet(x(0))$ for~\eqref{eq:parametricQP} is known for an initial condition $x(0)$,  the KKT conditions can equivalently be stated in the form
\REMARK{Is the equivalence of KKT conditions and~\eqref{eq:KKTparametricQP} shown in one of our early papers? Double check if the case with nonempty weakly active sets is included in the equivalence.}
\begin{equation}\label{eq:KKTparametricQP}
  \begin{split}
  \bigH\bigu + \bigG^\top\sigma+ \bigF^\top x(0)= 0 
  \\
  \bigG_\activeSet \bigu- \bigw_\activeSet- \bigE_\activeSet x(0)= 0
  \\
  \bigG_\inactiveSet \bigu- \bigw_\inactiveSet- \bigE_\inactiveSet x(0)\le 0
  \\ 
  \sigma_\inactiveSet= 0,\quad
  \sigma_\activeSet\ge 0
  \end{split}
\end{equation}
with multipliers $\sigma\in\R^q$ (see, e.g., \cite{Gupta2011}). 
They are solved by one of the affine functions that constitute the piecewise affine $\bigu_N^\star:\mathcal{F}_N\rightarrow\R^{mN}$ for one of the polytopes $\mathcal{P}_{N,i}$~\cite{Bemporad2002-Automatica}. 
It is therefore meaningful to say $\mathcal{A}$ defines a polytope $\mathcal{P}$ and the optimal $\bigu^\star_N(x(0))$ for all $x(0)\in\P$ and to refer to this polytope by $\P(\A)$. 
More precisely, $\P(\A)$ is defined as the relative interior of the set of initial states such that~\eqref{eq:KKTparametricQP} has a solution for the active set $\A$, which implies $\P(\A)$ is relative open. 
Because the solution to~\eqref{eq:OCP} is continuous, the affine law $x(0)\rightarrow \bigu_N^\star(x(0))$ can be extended from $\text{int}\,\P(\A)$ to the boundary of $\P(\A)$. 
Cumbersome statements about boundaries can be avoided with these definitions. 
For example, the interior of the central polytope in Figure~\ref{fig:PolytopicPartitions} and the optimal solution on its closure are defined by a single $\A$,
while eight different active sets (for 4 vertices and 4 facets), which all define the same optimal solution as $\A$, exist on its boundary. 
By using the notions "relative interior" and "relative openness"  
the statements in the paper carry over to lower-dimensional polytopes.\footnote{Proposition~\ref{prop:ActiveSetExtension}, for example, applies to active sets that define lower than $n$-dimensional polytopes.}
Since $\X$, $\U$ and $\T$ are used in the literature as defined in Section~\ref{sec:ProblemStatement} (and thus closed, full-dimensional, and their relative interiors are their interiors), we use the notation "$\text{int}\,\X$" etc.\ to refer to their interiors explicitly, while all other polytopes are understood to be open. 
\REMARK{we avoid stating the explicit expressions for $\mathcal{P}_{N,i}$ and $\bigu\star(x)$ on it, because they are not needed in the paper, and because discussing technicalities such as the row rank of $G_\activeSet$ and the possible non-uniqueness of the multipliers can be avoided (see~\cite{Bemporad2002-Automatica,Tondel2003-Automatica}).} 
%

Active sets of constraints are stated as tuples of bits. 
This proves to be convenient when considering the constraints stage by stage in~\eqref{eq:OCP} and the infinite-horizon problem.
For example, a tuple of $q$ bits $\a= (\a_1, \dots, \a_q)$ uniquely represents a set of active constraints 
$\mathcal{A}\subseteq\{1, \dots, q\}$, where 
\begin{equation}\label{eq:BitTupels}
  \a_i=\left\{\begin{array}{rl}
    1 & \text{if } i\in\A 
    \\
    0 & \text{otherwise}
  \end{array}\right.
\end{equation}
The concatenation of two or more tuples, say, $\a= (\a_1, \dots, \a_q)$ and $\a^\prime= (\a^\prime_1, \dots, \a^\prime_{q^\prime})$ is denoted and understood as
  $\a\a^\prime= (\a_1, \dots, \a_q, \a^\prime_1, \dots, \a^\prime_{q^\prime})$.
We use $\a$ and $\A$, $G_\a$ and $G_\A$, $\P(\A)$ and $\P(\a)$ etc.\ interchangeably. 

\section{The structure of the set of active sets}\label{sec:StructureOfSetOfActiveSets}
\subsection{Stagewise active set construction}\label{subsec:StagewiseConstruction}
We illustrated with Figure~\ref{fig:PolytopicPartitions} that 
the optimal feedback law and its polytopes for~\eqref{eq:OCP} with horizon $N$ are not contained in the law and polytopes for $N+1$. Such a property does hold for the active sets, however. This is stated more precisely in Proposition~\ref{prop:ActiveSetExtension}. 
As a preparation, the order of the constraints has to be agreed on. We stress that we can fix the order of the constraints without restriction, since 
the optimal solutions to~\eqref{eq:OCP} and \eqref{eq:parametricQP} are invariant to changing this order. 
Apart from the order stated in~\eqref{eq:OCPInputConstraints}--\eqref{eq:OCPTerminalConstraints}, it is natural to order the constraints stage by stage, i.e.,  
\begin{equation}\label{eq:ForwardConstraintOrder}
  \begin{split}
    x(0)&\in\X,\, u(0)\in\U
    \\
    x(1)&\in\mathcal{X},\, u(1)\in\mathcal{U}
    \\
    &\vdots
    \\
    x(N-1)&\in\X,\, u(N-1)\in\U
    \\
    x(N)&\in\T
  \end{split}
\end{equation}
where all but the last line correspond to $q_\X+ q_\U$ halfspace constraints, and the last line corresponds to $q_\T$ such constraints. 
\REMARK{Introduce \textit{stage}, where stage $N$ refers to the stage with the terminal constraints and stage $N-1$ to the last stage before it.}
\REMARK{Point out where these orders appear; some are natural in implementations. Note that some of them are useful for fast online implementations, but not so useful for EMPC.}
\REMARK{State: Note that the order within each stage does not affect the results of the paper.}

\REMARK{
    The following lemma could be stated before Proposition~\ref{prop:ActiveSetExtention}. The lemma essentially paraphrases the 'principle of optimality' for active sets. 
	\begin{lem}
		Consider the optimal control problem~\eqref{eq:OCP} for horizons $N$ and $N+1$. 
		Assume without restriction the constraints are ordered as in~\eqref{eq:ForwardConstraintOrder}. 
		Let $\a_{N+1}$ be an arbitrary active set of the problem with horizon $N+1$ and partition $\a_{N+1}$ according to
		\begin{equation*}
		  \a_{N+1}= \a\a_N,
		\end{equation*}
		where $\a$ and $\a_N$ have lengths $q_\X+ q_\U$ and $(N-1)(q_\X+ q_\U)$. Then, for any $x(0)\in\P_{N+1}(\a_{N+1})$, 
		the optimal solution $\bigu^\star(x^\star(1))$ is defined by $\a$ (and this statement can be iterated for $x^\star(l)$, $l\ge 1$, as in Corollary~\ref{cor:ShrinkingActiveSets}. 
	\end{lem}	
} 

\begin{prop}\label{prop:ActiveSetExtension}
Consider the optimal control problem~\eqref{eq:OCP} for horizons $N$ and $N+1$. 
Assume without restriction the constraints are  ordered as in~\eqref{eq:ForwardConstraintOrder}. 
Then for every active set $\a_{N+1}$ of the problem with horizon $N+1$ there exists an active set $\a_{N}$ for the problem with horizon $N$ such that  
  \begin{equation}\label{eq:ActiveSetExtension}
    \a_{N+1}= \a\a_{N} 
  \end{equation}
 for some $\a$ of length $q_{\X}+q_{\U}$.  
\end{prop}
\REMARK{Not a good idea to use both $\ell$ and $l$.}
\begin{proof}
  We introduce the abbreviations $\ell_\mathcal{T}(\xi)= \half \|\xi\|_P^2$, $\ell(\xi, \mu)=\half( \|\xi\|_Q^2+ \|\mu\|_R^2)$ and
  generalize~\eqref{eq:OCP} to $V^\star(\xi, [N_1, N_2])):=$
  \begin{subequations}\label{eq:GeneralizedOCP}
    \begin{align}
      \min\limits_{\substack{x(N_1+1), \dots, x(N_2),\\ u(N_1), \dots, u(N_2-1)}}\quad 
      & \ell_\mathcal{T}(x(N_2))+ \sum_{k=N_1}^{N_2-1}\ell(x(k), u(k))
      \end{align}
    subject to
    \begin{align}
      x(N_1)= \xi 
      \\
      x(k+1)= Ax(k)+ Bu(k)&,\, k= N_1, \dots, N_2-1  
      \label{eq:GeneralizedOCPDynamics}
      \\
      u(k)\in\mathcal{U}&,\, k= N_1, \dots, N_2-1
      \label{eq:GeneralizedOCPInputConstraints}
      \\
      x(k)\in\mathcal{X}&,\, k= N_1, \dots, N_2-1
      \\
      x(N_2)\in\mathcal{T} &
      \label{eq:GeneralizedOCPTerminalConstraints}
    \end{align}
  \end{subequations}
  for $N_2> N_1$. Since $A$, $B$, $P$, $Q$, $R$, $\U$, $\X$ and $\T$ are time-invariant, $V^\star(\xi, [N_1, N_2])$ only depends on $N_2-N_1$, i.e., 
  \begin{equation}\label{eq:TimeInvariance}
    V^\star(\xi, [N_1, N_2])= V^\star(\xi, [l, N_2-N_1+l])
  \end{equation}
  for all $l\in\N\cup\{0\}$ and all $\xi\in\mathcal{F}_{N_2-N_1}$. 
  Now consider the case $N_1= 0$, $N_2= N+1$, i.e., $V^\star(\xi, [0, N+1])$. 
  Expressing $V^\star(\xi, [0, N+1])$ in terms of $V^\star(\zeta, [1, N+1])$ in a fashion similar to backward dynamic programming yields
  $V^\star(\xi, [0, N+1])=$
  \begin{subequations}\label{eq:BackDynProgStep}
    \begin{align}\label{eq:BackDynProgStepCost}
      \min\limits_{x(1), u(0)}\Big(\ell(x(0), u(0))+ V^\star(\zeta, [1, N+1])\Big)
	\end{align}
	subject to
	\begin{equation}\label{eq:BackDynProgStepConstraints}\begin{split}
	  x(1)= Ax(0)+Bu(0),\, x(0)= \xi,\, \zeta= x(1)
	  \\
	  x(0)\in\X,\, u(0)\in\U
    \end{split}\end{equation}
  \end{subequations}
  Since $V^\star(\zeta, [1, N+1])= V^\star(\zeta, [0, N])$ according to~\eqref{eq:TimeInvariance}, and 
  with the notation~\eqref{eq:parametricQP} for $V^\star(\zeta, [0, N+1])$, \eqref{eq:BackDynProgStepCost} can be replaced by 
  \begin{equation}\label{eq:BackDynProgHelper1}
    \begin{split}
    &\min\limits_{x(1), u(0)} \bigg(\ell(x(0), u(0))+
    \\
    &\quad\min\limits_{\bigu} \Big(
      \half \zeta^\top \bigY \zeta+ \half \bigu^\top \bigH \bigu+  \zeta^\top \bigF \bigu \text{ s.t. }
      \bigG \lb{u}\le \bigw+ \bigE \zeta
    \Big)
    \bigg),
  \end{split}
\end{equation}
where $\bigu= (u^\top(1), \dots, u^\top(N))^\top$ here. 
Just as there exist matrices $\bigH$, $\bigY$ etc. that transform $V^\star(\zeta, [0, N+1])$ into~\eqref{eq:parametricQP}, there exist $Y$, $H$, $G$, $E$ and $w$ such that
\begin{equation}\label{eq:BackDynProgHelper2}
  \ell(x(0), u(0))= \half \xi^\top Y\xi+ \half u(0)^\top H u(0)
\end{equation}  
and~\eqref{eq:BackDynProgStepConstraints} can be stated as
\begin{equation}\label{eq:BackDynProgHelper3}\begin{split}
  Gu(0)&\le w+ E\xi
  \\
  \zeta&= A\xi+ Bu(0)
\end{split}\end{equation}
where~\eqref{eq:BackDynProgStepConstraints} comprises $q_\X+ q_\U$ constraints if $q_\X$ and $q_\U$ halfspaces define $\X$ and $\U$, respectively. 
\REMARK{Put the expressions for $Y$ etc. in a separate document for better documentation: $Y= Q$, $H=R$ and $Gu(0)\le E\xi+ w$ reads
\begin{equation*}
  \begin{bmatrix}
    E^\U \\ 0
  \end{bmatrix}
  u(0)\le
  \begin{bmatrix}
    0 \\ -E^\X
  \end{bmatrix}
  x(0)+
  \begin{bmatrix}
    v^\U \\ v^\X
  \end{bmatrix}
\end{equation*}
assuming $x\in\X$ and $u\in\U$ can be stated as $E^\X\le v^\X$ and $E^\U u\le v^\U$, respectively.} 
\REMARK{(Reader may ask why $F$ does not appear.)}
Combining~\eqref{eq:BackDynProgStep}--\eqref{eq:BackDynProgHelper3} yields
$V^\star(\xi, [0, N+1])=$
  \begin{subequations}\label{eq:OptProblemBackwardInductionStep}
    \begin{equation}\begin{split}
      \min\limits_{
	\red{\zeta},
	\blue{u(0)}, \green{\bigu}
      }\quad 
      & \left(
	  \blue{\half \xi^{\top}Y\xi + \half u(0)^{\top} H u(0)}
	\right. 
	\\
	&\left.
	  + \green{\half \zeta^\top \bigY \zeta+ \half \bigu^\top \bigH \bigu+ \zeta^\top \bigF \bigu}
	\right)
  \end{split}\end{equation}
  subject to
  \begin{align}
      & \blue{G u(0)\le w+ E\xi}\label{eq:OptProblemBackwardInductionStepConstraintsI}
      \\
      & \red{\zeta= A\xi+ Bu(0)}\label{eq:OptProblemBackwardInductionStepConstraintsII}
      \\
      & \green{\bigG \lb{u}\le \bigw+ \bigE \zeta}, \label{eq:OptProblemBackwardInductionStepConstraintsIII}
    \end{align}
  \end{subequations}
where the minimization with respect to~$\bigu$ can be applied to all terms of the cost function without restriction. 
\REMARK{Comment on necessity and sufficiency of the first order conditions.}
  The KKT conditions for~\eqref{eq:OptProblemBackwardInductionStep} read
    \begin{subequations}\label{eq:KKTInductionStep}
      \begin{align}
        &\begin{bmatrix}
          \bigY &  & \bigF 
          \\
           & H &  
          \\
          \bigF^\top & & \bigH 
        \end{bmatrix}
        \begin{bmatrix}
          \zeta \\ u(0) \\ \bigu
        \end{bmatrix}
        +
        \begin{bmatrix}
          & -\bigE^\top & I 
          \\
          G^\top & & -B^\top
          \\
          & \bigG^\top 
        \end{bmatrix}
        \begin{bmatrix}
          \lambda \\ \sigma \\ \tau 
        \end{bmatrix}
        = 0 
        \\
        \label{eq:JointInequalityConstraints}
    &\begin{bmatrix}
      G \\
      & \bigG
    \end{bmatrix}
    \begin{bmatrix}
      u(0) \\
      \bigu
    \end{bmatrix}
    - 
    \begin{bmatrix}
      E \\
      & \bigE
    \end{bmatrix}
    \begin{bmatrix}
	  \xi
	  \\
	  \zeta
    \end{bmatrix}
    -
    \begin{bmatrix}
      w\\
      \bigw
    \end{bmatrix}
    \le 0         
    \\
     & \zeta- A\xi- Bu(0)= 0
      \\
      &\lambda_i\left(Gu(0)-w-E\xi\right)_i= 0\text{ for all } i
      \\
      &\sigma_j\left(\bigG\bigu- \bigw- \bigE\zeta\right)_j = 0 \text{ for all } j
      \\
      &\lambda\ge 0
      \\
      &\sigma\ge 0,
      \end{align}
    \end{subequations}
where $\lambda$, $\tau$ and $\sigma$ are the multipliers for the conditions 
  \eqref{eq:OptProblemBackwardInductionStepConstraintsI},
  \eqref{eq:OptProblemBackwardInductionStepConstraintsII} and
  \eqref{eq:OptProblemBackwardInductionStepConstraintsIII}, respectively, $I$ denotes the unit matrix, and zero block matrices are omitted.
  \REMARK{Order of constraints in~\eqref{eq:OptProblemBackwardInductionStep} is different from order in optimality conditions.}
  Note that the inequality constraints~\eqref{eq:JointInequalityConstraints} are ordered as in~\eqref{eq:ForwardConstraintOrder}. 
  \REMARK{(Notation is not pretty anymore from here on, since $\alpha^-$ is now decorated with a minus while $\xi^-$ and $u^-$ no longer exist.)}
  Now let $\a_{N+1}$ be an arbitrary active set for~\eqref{eq:OptProblemBackwardInductionStep} 
  and let $\xi\in\mathcal{F}_{N+1}$ be an arbitrary initial condition such that $\a_{N+1}$ is the active set. 
  We partition $\a_{N+1}$ according to  
  \begin{equation}\label{eq:DesiredActiveSetPartition}
    \a_{N+1}= \a^-\lb{\a}, 
  \end{equation}
  where $\a^-$ corresponds to the $q_\X+ q_\U$ rows of $G$, $E$ and $w$ in~\eqref{eq:JointInequalityConstraints} and $\lb{\a}$ corresponds to the remaining rows in~\eqref{eq:JointInequalityConstraints}. 
  Let $\i^-$ and $\lb{\i}$ be the corresponding inactive sets. 
  Using these active and inactive sets, 
  the optimality conditions~\eqref{eq:KKTInductionStep} can equivalently be stated with separated active and inactive constraints 
  as in~\eqref{eq:KKTparametricQP}. 
  More precisely, there exist $\sigma^\star$, $\tau^\star$ and $\lambda^\star$ such that~\eqref{eq:KKTInductionStep} holds for 
  $\xi$, $u(0)^{\star}$, $\zeta^\star$, $\bigu^\star$, $\sigma^\star$, $\tau^\star$ and $\lambda^\star$ if and only if
  \begin{subequations}\label{eq:KKTBigProblem}
    \begin{align}
      & \red{
	\bigY\zeta^\star +\bigF\bigu^\star- \lb{E}^\top\sigma^\star+ \tau^\star = 0
      }
      \\
      & \blue{Hu(0)^{\star}+  G^\top_{\a^{-}}\lambda_{\a^{-}}- B^\top\tau^\star = 0}
      \\
      & \green{\bigH\bigu^\star+ \bigF^\top\zeta^\star + \bigG^\top_{\lb{\a}} \sigma^\star_{\lb{\a}}= 0}
      \label{eq:StationarityBigProblem}
      \\ 
      & \blue{\left(Gu(0)^{\star}-w-E\xi\right)_{\a^{-}} =0}
      \\
      & \green{\left(\bigG\bigu^\star- \bigw- \bigE\zeta\right)_{\lb{\a}} = 0}
      \label{eq:ActiveConstraintsBigProblem}
      \\
      & \blue{\left(Gu(0)^{\star}-w-E\xi\right)_{\i^-}\le 0}
      \\
      & \green{\left(\bigG\bigu^\star- \bigw- \bigE\xi^\star\right)_{\lb{\i}} \le 0}
      \label{eq:InactiveConstraintsBigProblem}
      \\
      & \red{\zeta^\star= A\xi+ Bu(0)^{\star}}
      \\
      & \blue{\lambda^\star_{\a^-}\ge 0, \quad \lambda^\star_{\i^-}= 0}
      \\
      & \green{\sigma^\star_{\lb{\a}}\ge 0, \quad \sigma^\star_{\lb{\i}}= 0}
      \label{eq:MultipliersBigProblem}
    \end{align}
  \end{subequations}
Since \eqref{eq:StationarityBigProblem}, \eqref{eq:ActiveConstraintsBigProblem}, \eqref{eq:InactiveConstraintsBigProblem}, 
\eqref{eq:MultipliersBigProblem} are the optimality conditions of \eqref{eq:parametricQP}, 
i.e., of $V^\star(\zeta, [0, N])$,
we have $\zeta= x(1)\in\F_N$ and 
$\lb{\a}$ is an active set of $V^\star(\zeta, [0, N])$. 
Since $\a_{N+1}$ was an arbitrary active set of~\eqref{eq:OptProblemBackwardInductionStep} and since $\a$ can be partitioned as in~\eqref{eq:DesiredActiveSetPartition}, claim~\eqref{eq:ActiveSetExtension} holds with 
$\alpha= \alpha^-$, $\lb{\a}= \a_{N}$. 
\end{proof}

\REMARK{Think about generalizations of Proposition~\ref{prop:ActiveSetExtension}. For example, if backward constraint order is assumed, \eqref{eq:ActiveConstraintHelper} is replaced by
  \begin{equation}\label{ActiveConstraintHelper2}
    \begin{bmatrix}
      & \bigG \\
      G
    \end{bmatrix}
    \left(\begin{array}{l}
      u^{-\star} \\
      \bigu^\star 
    \end{array}\right)
    \le
    \begin{bmatrix}
      & \bigE \\
      E
    \end{bmatrix}
    \left(\begin{array}{l}
	  \xi^-
	  \\
	  \xi
    \end{array}\right)
    +
    \begin{bmatrix}
      \bigw\\
      w
    \end{bmatrix},
  \end{equation}
  and the partition $\a= \lb{\a}\a^-$ is chosen instead of~\eqref{eq:DesiredActiveSetPartition}.  
} 
\REMARK{(Think about generalizations of Proposition~\ref{prop:ActiveSetExtension} that only require an assumption on the constraint order in the proof but not the claim itself. See TODO at this position in the text.)}
A shorter proof of Proposition~\ref{prop:ActiveSetExtension} can be stated based on~\eqref{eq:GeneralizedOCP}, \eqref{eq:TimeInvariance} and by referring to the the principle of optimality in dynamic programming (see, e.g., \cite[chapter 1.3]{Bertsekas2005}). 
In particular, the optimality conditions~\eqref{eq:KKTparametricQP} and~\eqref{eq:KKTInductionStep} would not be required in this case. 
We state the more detailed proof, because it shows that constraint qualifications and weakly active constraints play no role for Proposition~\ref{prop:ActiveSetExtension}, while they often result in special cases elsewhere
(see, e.g., 
\cite[section 3.4.1]{Gupta2011},
\cite[section 5]{Tondel2003-Automatica},
\cite[section 4.1.1]{Bemporad2002-Automatica}). A case where the linear independence constraint qualification 
(\cite[Section 5.2.1]{Bazaraa2006}, \cite[section 4.1.1]{Bemporad2002-Automatica})
fails but Proposition~\ref{prop:ActiveSetExtension} still applies is given in Example~\ref{ex:ActiveSetExtension}, part (ii). 

It is easy to see that truncating an active set by removing stages "on the left" as in~\eqref{eq:ActiveSetExtension} results in active sets that define the optimal solution to~\eqref{eq:OCP} on a shrinking horizon: 
\REMARK{Incorporate case $l= N$ into Cor.}
\begin{cor}\label{cor:ActiveSetsShrinkingHorizon}
  Consider~\eqref{eq:OCP} for horizon $N$ and assume the constraints to be ordered as in~\eqref{eq:ForwardConstraintOrder} without restriction.
  Let $\a_N$ be an arbitrary active set and partition $\a_N$ according to
  \begin{equation*}
  \a_{N}= \underbrace{\a_{N,0}}_{q_\X+ q_\U}
  \cdots\,
  \underbrace{\a_{N,N-1}}_{q_\X+ q_\U}
  \underbrace{\a_{N,N}}_{q_\T}. 
  \end{equation*}
  Let $x(0)$ be an arbitrary initial condition that results in the active set $\a_N$ and let 
  $x^\star(k)$, $k\ge 0$ introduced in~\eqref{eq:InfiniteOptimalSequences} refer to the optimal solution for $x(0)$.    
  Then, for all $l\in\{0, \dots, N-1\}$, the active set 
  \begin{equation*}
    \a_{N-l}= \a_{N,l}\dots\a_{N,N}
  \end{equation*}
  defines the optimal solution for~\eqref{eq:OCP} with horizon $N-l$ and initial condition $x^\star(l)$.  
\end{cor}
\begin{proof}
It suffices to prove the claim for $l= 1$ and to apply this case repeatedly. 
We showed in the proof of Proposition~\ref{prop:ActiveSetExtension} that 
$\a_N$ is the active set for the optimal successor state $\zeta= x^\star(1)$ and~\eqref{eq:OCP} with horizon $N$,
if $\a_{N+1}$ is the active set for $x(0)$ and horizon $N+1$. With the substitutions $N+1\rightarrow N$ and $N\rightarrow N-1$, the claim for $l= 1$ results. 
\end{proof}

Corollary~\ref{cor:ActiveSetsShrinkingHorizon} essentially extends remark (i) made in the notation section from the optimal sequences of inputs and states~\eqref{eq:InfiniteOptimalSequences} to active sets. The corresponding extension to polytopes appears in part (i) of Remark~\ref{rem:Polytopes}. 

Proposition~\ref{prop:ActiveSetExtension} and Corollary~\ref{cor:ActiveSetsShrinkingHorizon} are illustrated with Example~\ref{ex:ActiveSetExtension}. 
All active sets and their properties stated in Example~\ref{ex:ActiveSetExtension} can be checked with simple calcuations, which are not stated here.
\REMARK{The active set $\a_{N+1,j}$ in Proposition~\ref{prop:ActiveSetExtension} may respect the linear independence constraint qualification, while $\a_{N,i}$ does not. Consequently, $\a_{N+1,j}$ may define a full-dimensional polytope in the solution for horizon $N+1$, while the polytope for horizon $N$ defined by $\a_{N,i}$ is not full-dimensional. This is illustrated in Example~\ref{ex:ActiveSetExtension}.}
\REMARK{(There is a longer discussion of the LICQ case in the code at the end of the subsection.)}
\begin{example}\label{ex:ActiveSetExtension}
	Consider~\eqref{eq:OCP} for 
	\begin{equation*}
	  A= \begin{bmatrix}
	    -\half & \half 
	    \\
	    -\half & -\half 
	  \end{bmatrix},
	  \,
	  B= \begin{bmatrix}
	    1 \\ 1
	  \end{bmatrix},
	  \,
	\end{equation*}
	$\X= \{x\in\R^2; -10\le x_i\le 10, i= 1, 2\}$,
	$\U= \{u\in\R; -1\le u\le 1\}$, $Q$ is the identity matrix, $R= 0.1$ and $P$ and $\T$ are as in Section~\ref{sec:ProblemStatement}.  
	Note that $q_\X+ q_\U= 6$, and $q_\T= 4$ results for this example. 
	Figure~\ref{fig:ActiveSetExtension} shows the polytopes for $N= 1$ and $N= 2$, which illustrate the following cases of Proposition~\ref{prop:ActiveSetExtension}:
	
	(i) Active sets for horizon $N$ may be extended to exactly one, more than one, or no active set at all. 
	%
	For example, the following active set for $N= 1$ is extended to exactly one active set for $N= 2$: 
	\begin{equation*}
	  \begin{tabular}{lrr}
	  & $000000.0001$ & $(N= 1)$
	  \\
	  & $100000.000000.0001$ & $(N= 2)$
	  \end{tabular}
	\end{equation*}
	(green polytopes in Figure~\ref{fig:ActiveSetExtension}; dots are introduced every $q_\X+ q_\U= 6$ positions for convenience). 
	The following active set for $N=1$ is extended to three active sets for $N=2$:
	\begin{equation}\label{eq:ActiveSetExtensionSampleAs}
	  \begin{tabular}{lrr}
	  &  $100000.0000$ & $(N=1)$
	  \\
	  & $000000.100000.0000$ & $(N=2)$
	  \\
	  & $100000.100000.0000$ & $(N=2)$
	  \\
	  & $010000.100000.0000$ & $(N=2)$
	  \end{tabular}
	\end{equation}
	(yellow polytopes). 
	The active set 
	\begin{equation*}
		\begin{tabular}{lrr}
		 & $100000.000000.0001$ & $(N=2)$
		\end{tabular}
	\end{equation*}
	is not extended to any set for $N= 3$ (green polytope for $N=2$ in Figure~\ref{fig:ActiveSetExtension}b; $N= 3$ not shown). 
	
	(ii) Active sets that respect the linear independence constraint qualification (licq)  for $N+1$ may result from extending an active set that does not (or vice versa). 
	For example, the following active sets appear in the example
	\begin{equation*}
		\begin{tabular}{lrr}
		& $010000.0001$ & $(N=1)$
		\\
		& $000000.010000.0001$ & $(N=2)$
		\end{tabular}
	\end{equation*}
	and simple calculations show that the latter does and the former does not respect licq (blue polytopes). 
	
	We state the active sets in set notation for completeness. They read
	$\{10\}$,
	$\{1, 16\}$,
	$\{1\}$,
	$\{7\}$,
	$\{1, 7\}$,
	$\{2, 7\}$,
	$\{1, 16\}$,
	$\{2, 10\}$,
	$\{8, 16\}$ in the order they appear in above.
\end{example}
\begin{figure}
  \centering
  \subfloat[$N=1$]{\scriptsize{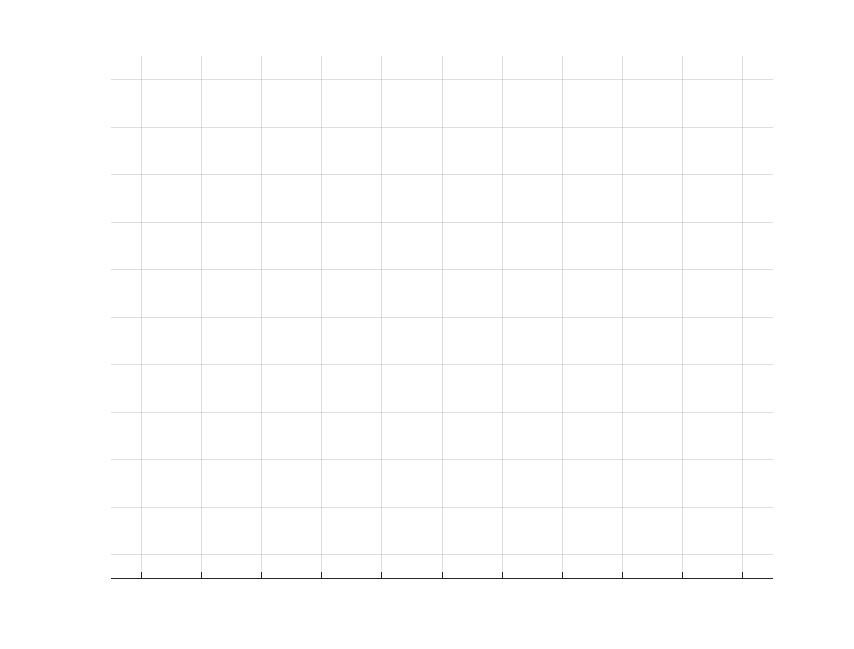}}\qquad
  \subfloat[$N=2$]{\scriptsize{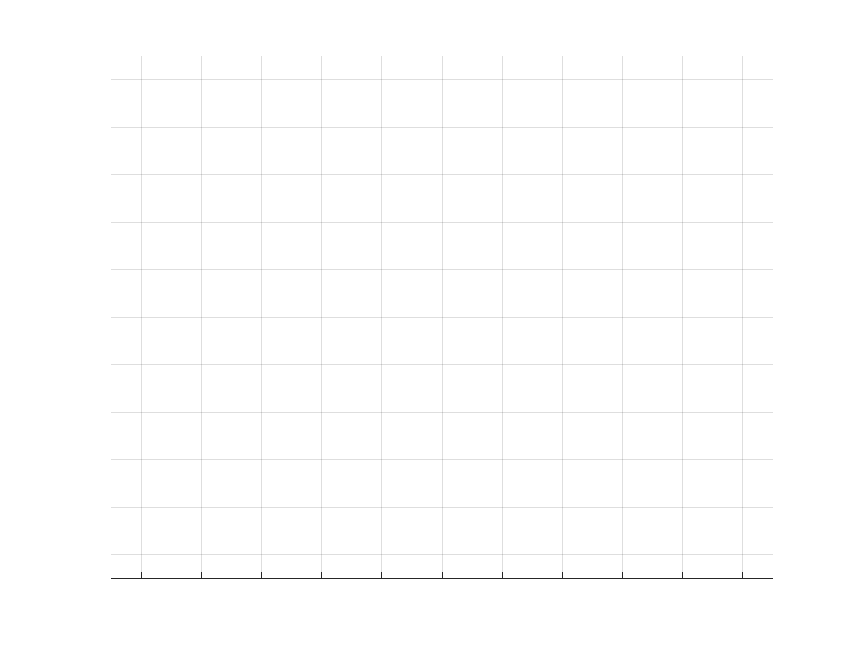}}\qquad
  \caption{Partitions of the solutions to~\eqref{eq:OCP} for Example~\ref{ex:ActiveSetExtension}.}
  \label{fig:ActiveSetExtension}
\end{figure} 

Proposition~\ref{prop:ActiveSetExtension} and Corollary~\ref{cor:ActiveSetsShrinkingHorizon} 
do not require the assumptions on $P$, $K_\infty$ and $\T$ stated in Section~\ref{sec:ProblemStatement} to hold 
and therefore apply to a larger problem class.   
All statements in Section~\ref{sec:Persistency}, in contrast, do require the assumptions on $P$, $K_\infty$ and $\T$ stated in Section~\ref{sec:ProblemStatement} to hold.  
Example~\ref{ex:ActiveSetExtension} respects the assumptions on $P$, $K_\infty$ and $\T$ just so the same example can be used throughout the paper. 


\subsection{Persistency of active sets, polytopes and optimal solutions}
\label{sec:Persistency}
Whenever a polytope and the optimal solution on it are defined by an active set for which all terminal constraints are inactive, this polytope and the optimal solution do not change when the horizon is extended. The following lemma and proposition state this more precisely. 
\begin{lem}\label{lem:PersistentActiveSets}
	Consider the optimal control problem~\eqref{eq:OCP} 
	and assume without restriction the forward constraint order~\eqref{eq:ForwardConstraintOrder} is used.
	Let $\tilde{\a}$ be an arbitrary index set with length $N(q_\X+ q_\U)$ and let $l\ge 0$ be arbitrary.
	The index set 
	\begin{equation}\label{eq:PersistentActiveSetsHelper1}
	  \tilde{\a}\underbrace{0\dots 0}_{q_\T}.
	\end{equation}
	is an active set for horizon $N$ if and only if the active set 
	\begin{equation}\label{eq:PersistentActiveSetsHelper2}
	  \tilde{\a} 
	  \underbrace{\underbrace{0\cdots 0}_{ q_\X+ q_\U}
	    \cdots
	    \underbrace{0\cdots 0}_{ q_\X+ q_\U}
	  }_{l} 
	  \underbrace{0\dots 0}_{q_\T}
	\end{equation}
	is an active set for horizon $N+l$. 
\end{lem}
\begin{proof}
	It suffices to prove the claim for $l= 1$, since the cases $l>1$ follow by induction.  
	To show an active set~\eqref{eq:PersistentActiveSetsHelper1} implies the existence of~\eqref{eq:PersistentActiveSetsHelper2}, 
	let $x(0)\in\F_N$ be an arbitrary initial condition that results in~\eqref{eq:PersistentActiveSetsHelper1}.
	Since~\eqref{eq:PersistentActiveSetsHelper1} implies inactivity of the terminal constraints, we have 
	\begin{equation}\label{eq:PersistentActiveSetsHelper25}
	  x^\star(N)\in\text{int}\,\mathcal{T}.
	\end{equation}
	Consequently, \eqref{eq:Lemma2.2Chmielewski1996} applies and the infinite-horizon problem and the finite-horizon problem for horizons $N$, $N+1$ have the same optimal solution, which we denote	
%
	\begin{equation}\label{eq:PersistentActiveSetsHelper24}
	  u^\star(k), \, x^\star(k), \, k\ge 0.
	\end{equation}
    It remains to prove the active set for $x(0)$ and horizon $N+1$ has the form~\eqref{eq:PersistentActiveSetsHelper2} for this solution,
    which can be done by showing $x^\star(N)\in\text{int}\,\X$, $u^\star(N)\in\text{int}\,\U$ and $x^\star(N+1)\in\text{int}\,\T$.
	It is easy to show that~\eqref{eq:PersistentActiveSetsHelper25} implies 
	\begin{equation}\label{eq:PersistentActiveSetsHelper19} 
	  x^\star(N+l)\in\text{int}\,\T\text{ for all }l\ge 0 
	\end{equation}
	due to the positive invariance of $\T$ (see Lemma~\ref{lem:PositiveInvarianceAndInterior} in the appendix). 
	This in particular implies 
   \begin{equation}\label{eq:PersistentActiveSetsHelper20}
     x^\star(N+1)\in\text{int}\,\T .
   \end{equation}	
	Since $x^\star(N)\in\text{int}\,\T$, and since $\text{int}\,\T\subseteq\text{int}\,\X$ follows from
	$\T\subseteq\X$, we also have
	\begin{equation}\label{eq:PersistentActiveSetsHelper21}
	  x^\star(N)\in\text{int}\,\X.
	\end{equation}
	\REMARK{$\lambda$ is used both as multiplier and scaling factor:}
    The proof for $u^\star(N)\in\text{int}\,\U$ is somewhat technical.
    It is easy to show that there exists a $\lambda\in[0, 1)$ such that $x^\star(N)\in\lambda\T:= \{\lambda \xi|\xi\in\T\}$ and $\lambda\T\subset\text{int}\,\T$, because $x^\star(N)\in\text{int}\,\T$, $0\in\text{int}\,\T$~\cite{Sznaier1987} and $\T$ is convex~\cite{Gutman1987}.
    \REMARK{Alternatively, let $\lambda= \psi_{\T}(x^\star(N))$, where $\psi_{\T}$ is the Minkowski function for $\T$, and $\T$ and $\lambda\T$ are two sublevel sets of $\psi_{\T}$ with $\lambda\T\subset\text{int}\,\T$ (see, e.g., \cite[chapter 3.1.2]{Blanchini}).}   
    Now $x^\star(N)\in\lambda\T$ implies there exists a $\xi\in\T$ such that $x^\star(N)= \lambda\xi$. 
    By definition of $\T$, $\xi\in\T$ implies $\xi\in\X_{\U}$ and by definition of $\X_{\U}$ we have  $K_\infty \xi\in\U$. 
    From
    \begin{equation*}
      \frac{1}{\lambda}K_\infty x= K_\infty \frac{1}{\lambda}\lambda\xi= K_\infty\xi\in\U
    \end{equation*}
    we infer $K_\infty x^\star(N)\in\lambda\U$. 
    Together with $0\in\text{int}\,\U$ and the convexity of $\U$ this implies $K_\infty x^\star(N)\in\text{int}\,\U$.
    Therefore
	\begin{equation}\label{eq:PersistentActiveSetsHelper22}
	  u^\star(N)= K_\infty x^\star(N)\in\text{int}\,\U.
	\end{equation}

To show an active set~\eqref{eq:PersistentActiveSetsHelper2} implies the existence of~\eqref{eq:PersistentActiveSetsHelper1}, 
let $l= 1$,
let $x(0)$ be an arbitrary initial condition that results in~\eqref{eq:PersistentActiveSetsHelper2},
and let~\eqref{eq:PersistentActiveSetsHelper24} denote the optimal solution for horizon $N+1$, which now needs to be shown to be optimal for $N$.
The trailing zeros in~\eqref{eq:PersistentActiveSetsHelper2} imply the constraints in stages $N$ and $N+1$ are inactive,
i.e.,
\begin{equation}\label{eq:PersistentPolytopesHelper20}
  \begin{split}
	  x^\star(N)\in\text{int}\,\X, \,
	  u^\star(N)\in\text{int}\,\U,
	  \\
	  x^\star(N+1)\in\text{int}\,\T\subseteq\text{int}\,\X.
  \end{split}
\end{equation} 
By the same arguments as for~\eqref{eq:PersistentActiveSetsHelper19} we have
\begin{equation}\label{eq:PersistentPolytopesHelper21}
  x^\star(N+l)\in\text{int}\,\T\subseteq\text{int}\,\X\text{ for all }l\ge 0.
\end{equation}
Together~\eqref{eq:PersistentPolytopesHelper20} and~\eqref{eq:PersistentPolytopesHelper21} imply that~\eqref{eq:OCP} for horizon $N=1$ and initial condition $x^\star(N)$ has the same solution as the unconstrained infinite-horizon problem.
Consequently, 
$x^\star(N)\in\T$, since $\T$ is the largest set of initial conditions such that \eqref{eq:OCP} and the unconstrained problem result in the same solution.
Since $x^\star(N)\in\T$, the input sequence and trajectory~\eqref{eq:PersistentActiveSetsHelper24} with initial condition $x(0)$ are feasible for horizon $N$
and $u^\star(N)= K_\infty x^\star(N)$. 

We need to show $x^\star(N)\in\text{int}\,\mathcal{T}$ to complete the proof, which can be done by showing the existence of an open ball centered at $x^\star(N)$ in $\mathcal{T}$. This requires three preparations.
First note that $x^\star(N+1)\in\text{int}\,\T$ implies there exists an $\epsilon_1>0$ such that 
\begin{equation}\label{eq:PersistentPolytopesHelper39}
  B_{\epsilon_1}(x^\star(N+1))\subset\T. 
\end{equation}
By the definition of $\T$ in~\eqref{eq:TerminalSet} we have
\begin{equation}\label{eq:PersistentPolytopesHelper40}
  (A+BK_\infty)^k \xi\in\X_\U \,\forall\, k\ge 0, \, \xi\in B_{\epsilon_1}(x^\star(N+1))
\end{equation}
Secondly, $u^\star(N)\in\text{int}\,\U$ and $x^\star(N)\in\text{int}\,\X$ imply
there exist $\epsilon_2>0$, $\epsilon_3>0$ such that $B_{\epsilon_2}(u^\star(N))\subset \U$ and
$B_{\epsilon_3}(x^\star(N))\subset\X$. 
Moreover, $\epsilon_3$ can be chosen sufficiently small to ensure 
$K_\infty B_{\epsilon_3}(x^\star(N))\subset B_{\epsilon_2}(u^\star(N))\subset\U$, because $x\rightarrow K_\infty x$ is linear. 
\REMARK{Check which property is actually necessary here instead of linearity.}
Together $B_{\epsilon_3}(x^\star(N))\subset\X$ and $K_\infty B_{\epsilon_3}(u^\star(N))\subset\U$ yield
\begin{equation}\label{eq:PersistentPolytopesHelper41}
  B_{\epsilon_3}(x^\star(N))\subset\X_\U
\end{equation}
by definition of $\X_\U$ (see~\eqref{eq:TerminalSet}). 
As a third preparation, note that $\epsilon_3$ can be chosen sufficiently small to ensure
\begin{equation}\label{eq:PersistentPolytopesHelper42}
  (A+BK_\infty) B_{\epsilon_3}(x^\star(N))\subset B_{\epsilon_1}(x^\star(N+1)),
\end{equation}
because $x\rightarrow(A+BK_\infty)x$ is linear. 
By collecting intermediate results we find, for all $\xi\in B_{\epsilon_3}(x^\star(N))$,
\begin{equation*}
  \begin{split}
    (A+BK_\infty)^0\xi&= \xi\in\X_\U \text{ (acc. to~\eqref{eq:PersistentPolytopesHelper41})}
    \\
    (A+BK_\infty)^1\xi&\in \X_\U
    \text{ (acc. to~\eqref{eq:PersistentPolytopesHelper42}, \eqref{eq:PersistentPolytopesHelper40})}
    \\
    (A+BK_\infty)^l(A+BK_\infty)\xi&\in\X_\U\,\forall l\ge 0
    \text{ (acc. to~\eqref{eq:PersistentPolytopesHelper42}, \eqref{eq:PersistentPolytopesHelper40})}
  \end{split}
\end{equation*}
Together these three statements imply $B_{\epsilon_3}(x^\star(N))\subset\T$ by definition of $\T$ in~\eqref{eq:TerminalSet}, or equivalently, 
$x^\star(N)\in\text{int}\,\T$. This proves $x(0)$ results in the active set~\eqref{eq:PersistentActiveSetsHelper1} for horizon $N$. Furthermore, $x^\star(N)\in\text{int}\,\T$ implies that \eqref{eq:Lemma2.2Chmielewski1996} applies. Consequently, \eqref{eq:PersistentActiveSetsHelper24} is also the optimal solution for horizon $N$. 
\end{proof}
\REMARK{\begin{rem}
Since the proof of Proposition~\ref{prop:ActiveSetExtension} is carried out with the optimality conditions, it is obvious to ask whether the proof of Lemma~\ref{lem:PersistentActiveSets} can be carried out in the same manner. This would have to result in expressions for two polytopes that turn out to be identical. To this end, the cost function of the OCP with horizons $N$ and $N+1$ have to be shown to be identical for $x(0)$ such that $x(N)\in\T$. See \textit{Notes.tex}. 
\end{rem}} 

It remains to show that the active sets~\eqref{eq:PersistentActiveSetsHelper1} and~\eqref{eq:PersistentActiveSetsHelper2} define the same polytope and the same optimal solution on it. We call polytopes with this property persistent.
\REMARK{Relate 'polytope $\P$' in the definition to a pwa control law and its polytopes. This way, it is stated, implicitly at least, that 'polytope' refers to a single polytope, i.e., a polytope that is defined by a single active set and not a union of two or more such polytopes that maybe turns out to be a convex union by chance.}
\REMARK{Think about moving this definition further up front, around the first figure.}
\begin{defn}\label{def:PersistentActiveSets}
A Polytope $\mathcal{P}$ and the optimal solution on it are called \textit{persistent from horizon $N$ on}, if the polytope exists for all horizons $N+l$, $l\ge0$ and the optimal solution for the optimal control problem~\eqref{eq:OCP} remains the same for all initial conditions $x(0)\in\mathcal{P}$ and all $l\ge0$. 
\end{defn}
We omit "from horizon $N$ on" when $N$ is obvious. Note that $N$ is not necessarily the smallest possible one in the definition. 
The term "persistent active set" refers to any active set that defines a persistent polytope.  
\REMARK{We must distinguish the set of all persistent polytopes from the set of persistent polytopes with active sets of the form~\eqref{eq:PersistentActiveSetsHelper1}.}
\REMARK{
Prove the converse of Cor.~\ref{cor:SetOfPersistentPolytopes} or prove it does not hold. (Is every persistent polytope defined by an active set of the form~\eqref{eq:PersistentActiveSetsHelper19}?)
}
\begin{prop}\label{prop:PersistentPolytopes}
Assume the conditions stated in Lemma~\ref{lem:PersistentActiveSets} to hold. Then the polytope $\P$ and solution on it defined 
by the active set~\eqref{eq:PersistentActiveSetsHelper1} are persistent from horizon $N$ on. 
Furthermore, for any horizon $N+l$, $l\ge 0$ and the optimal solution on it are defined by the active set~\eqref{eq:PersistentActiveSetsHelper2}.
\end{prop}
\begin{proof}
Let $l\ge0$ be arbitrary and let $\mathcal{P}_N$ and $\mathcal{P}_{N+l}$ refer to the polytopes defined by the active sets~\eqref{eq:PersistentActiveSetsHelper1} and~\eqref{eq:PersistentActiveSetsHelper2}, respectively. 
Any $x(0)$ with active set~\eqref{eq:PersistentActiveSetsHelper1} for horizon $N$ results in the active set~\eqref{eq:PersistentActiveSetsHelper2} for horizon $N+l$ according to the first part of the proof of Lemma~\ref{lem:PersistentActiveSets}. 
This implies $\mathcal{P}_{N}\subseteq\mathcal{P}_{N+l}$. 
Analogously, the second part of the proof implies $\mathcal{P}_{N+l}\subseteq\mathcal{P}_N$. 
The proof of Lemma~\ref{lem:PersistentActiveSets} also established the equality of the input sequences and optimal trajectories~\eqref{eq:PersistentActiveSetsHelper24} for all  $x(0)\in\P_N= \P_{N+l}$. 
\end{proof}
The following example illustrates Lemma~\ref{lem:PersistentActiveSets} and Proposition~\ref{prop:PersistentPolytopes}. 
\begin{figure}
  \centering
  \scriptsize{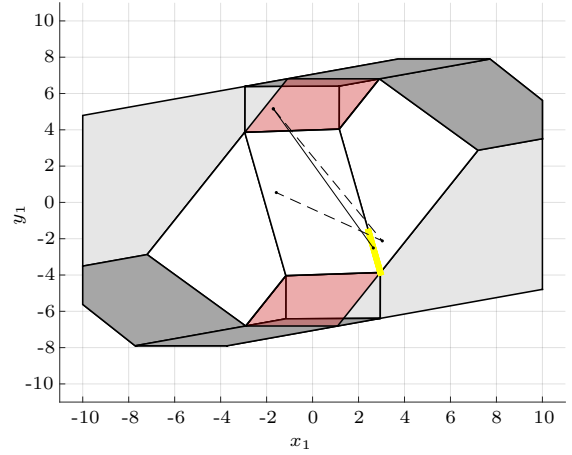}
  \caption{Illustrations for Example~\ref{ex:PersistentActiveSets}. 
  The one-dimensional yellow polytope (width increased for better visibility only) marks $x^\star(1)$ for $N=1$ and all $x(0)$ in the upper red polytope. }
  \label{fig:PersistentActiveSets}
\end{figure} 
\begin{example}\label{ex:PersistentActiveSets}
Consider the same problem as in Example~\ref{ex:ActiveSetExtension}. 
The white polytopes in Figure~\ref{fig:PersistentActiveSets} are defined by the active sets
\begin{equation*}
  \begin{tabular}{rcr}
  \multicolumn{1}{c}{$N=1$} & & \multicolumn{1}{c}{$N=2$}
  \\\hline
   $000000.0000$ &&  $000000.000000.0000$ \\ 
   $100000.0000$ &&  $100000.000000.0000$ \\ 
   $010000.0000$ &&  $010000.000000.0000$ 
  \end{tabular}
\end{equation*}
which are persistent polytopes from $N=1$ and $N= 2$ on, respectively, according to Proposition~\ref{prop:PersistentPolytopes}.

The two red polytopes for $N= 1$ shown in Figure~\ref{fig:PersistentActiveSets} result for the active sets
\begin{equation*}
  \begin{tabular}{r}
    $000000.0001$\\
    $000000.0010$
  \end{tabular} 
\end{equation*}
which do not respect the conditions of Proposition~\ref{prop:ActiveSetExtension}. 
Figure~\ref{fig:PersistentActiveSets} shows that they are not persistent but disappear for $N= 2$.  
 
The light grey polytopes in Figure~\ref{fig:PersistentActiveSets} are defined by the active sets
\begin{equation}\label{eq:Example2ActiveSetsPersistentFrom2On}
  \begin{tabular}{r}
    $000000.100000.0000$ \\ 
    $100000.100000.0000$ \\ 
    $010000.100000.0000$ \\ 
    $000000.010000.0000$ \\ 
    $100000.010000.0000$ \\ 
    $010000.010000.0000$ \\ 
  \end{tabular}
\end{equation}
for $N= 2$. These polytopes are persistent from $N= 2$ on according to Proposition~\ref{prop:PersistentPolytopes}.

The dark grey polytopes in Figure~\ref{fig:PersistentActiveSets} are defined by the active sets
\begin{equation}\label{eq:Example2ActiveSetsNotPersistent}
  \begin{tabular}{r}
    $100000.000000.0001$ \\ 
    $010000.000000.0010$ \\ 
    $000000.010000.0010$ \\ 
    $000000.100000.0001$ 
  \end{tabular}
\end{equation}
for $N= 2$. An analysis of the example for $N= 3$ (not detailed here) shows that the polytopes defined by~\eqref{eq:Example2ActiveSetsPersistentFrom2On} and~\eqref{eq:Example2ActiveSetsNotPersistent} are indeed persistent from $N= 2$ on and not persistent, respectively. 

Finally, Figure~\ref{fig:PersistentActiveSets} shows the solution to~\eqref{eq:OCP} 
for a sample initial condition for $N= 1$ and $N= 2$. 
The two trajectories illustrate that optimal trajectories change if the horizon is increased for initial conditions from non-persistent polytopes (cp.\ remark (iii) and the subsequent discussion in the notation section). 
Moreover, $x^\star(N)\in\partial\T$ results for $N=1$, and since the trajectories for $N= 1$ and $N=2$ differ, this shows $x^\star(N)\in\T$ is not sufficient in~\eqref{eq:Lemma2.2Chmielewski1996} and Lemma~\ref{lem:PersistentActiveSets}.
\end{example}
\REMARK{Proposition~\ref{prop:PersistentPolytopes} requires $\a_N$ to end with $q_\T$ zeroes. This implies $x^\star(N)\in\text{int}\,\T$ (see~\eqref{eq:PersistentActiveSetsHelper25}). The following example shows $x^\star(N)\in\T$ is in general not sufficient, 
which essentially results, because $x\in\text{int}\,\T$ can in general not be replaced by $x\in\T$ in~\eqref{eq:Lemma2.2Chmielewski1996}.}
Lemma~\ref{lem:PersistentActiveSets} and Proposition~\ref{prop:PersistentPolytopes} essentially result from inserting zeroes, i.e., stages with inactive constraints, "on the right" of an active set of the form~\eqref{eq:PersistentActiveSetsHelper1}. Just as in Corollary~\ref{cor:ActiveSetsShrinkingHorizon}, we can also remove stages "on the left" to obtain new active sets. In general this results in active sets for a \textit{shrinking} horizon as in Corollary~\ref{cor:ActiveSetsShrinkingHorizon}. For persistent active sets~\eqref{eq:PersistentActiveSetsHelper1}, new active sets for the \textit{same} horizon result, and they are persistent themselves:  
\REMARK{The following Cor. (and possibly subsequent statements) should more naturally include the case $l= N$.}
\begin{cor}\label{cor:PersistentPolytopesShrinkingHorizon}
  Assume $\P$ is a polytope that is persistent from horizon $N$ on with an 
  active set $\a_N$ of the form~\eqref{eq:PersistentActiveSetsHelper1}. 
  Let $\a_{N,i}$ be tuples of length $q_\X+ q_\U$ such that 
  \begin{align}\label{eq:SequenceOfPersistentPolytopesHelper1}
		  \a_N= \a_{N,0} \cdots \a_{N,N-1}\underbrace{0\cdots 0}_{q_\T}.
  \end{align}
  Then the polytope defined by 
  \begin{align}\label{eq:SequenceOfPersistentPolytopesHelper2}
    \a_{N-l}= \a_{N, l} \cdots \a_{N,N-1}\underbrace{0\cdots 0}_{q_\T}
  \end{align}
  is persistent from $N-l$ on for any $l\in\{0, \dots, N-1\}$. 
  Moreover, for any such $l$, the polytope defined by~\eqref{eq:SequenceOfPersistentPolytopesHelper2} for horizon $N-l$ 
  is defined by
  \begin{align}\label{eq:SequenceOfPersistentPolytopesHelper3}
    \tilde{\a}_{N}= \a_{N, l} \cdots \a_{N,N-1}
    \underbrace{0\dots0}_{l\cdot(q_\X+ q_\U)}
    \underbrace{0\cdots 0}_{q_\T}
  \end{align}   
  for horizon $N$ and persistent from horizon $N$ on. 
\end{cor}
\REMARK{It needs to be clearer that there result $N+1$ new persistent polytopes in the previous Corollary.}
\begin{proof}
  Let $x(0)\in\P$ be arbitrary and let $u^\star(k)$, $x^\star(k)$, $k\ge 0$ be as in~\eqref{eq:InfiniteOptimalSequences}. 
  According to Corollary~\ref{cor:ActiveSetsShrinkingHorizon}, the active set~\eqref{eq:SequenceOfPersistentPolytopesHelper2} defines the optimal solution and polytope for $x^\star(l)$. 
  According to Proposition~\ref{prop:PersistentPolytopes}, 
  the active set~\eqref{eq:SequenceOfPersistentPolytopesHelper2} defines a polytope that is persistent from horizon $N-l$ on, 
  therefore it is persistent from $N$ on. 
  The claim about~\eqref{eq:SequenceOfPersistentPolytopesHelper3} follows by applying Lemma~\ref{lem:PersistentActiveSets} to~\eqref{eq:SequenceOfPersistentPolytopesHelper2}.  
\end{proof}

It is evident from Figure~\ref{fig:ActiveSetExtension} that the set of persistent polytopes (white and yellow polytopes in the top figure) is in general not convex. 
The following two corollaries summarize some other properties of the set of persistent polytopes 
with active sets~\eqref{eq:PersistentActiveSetsHelper1}. 
\begin{cor}\label{cor:SetOfPersistentPolytopes}
    Let $\mathbb{P}_N$ refer to the union of all persistent polytopes with active sets of the form~\eqref{eq:PersistentActiveSetsHelper1}.
	$\mathbb{P}_N$ is positive invariant under 
	the open-loop optimally controlled system, i.e., $x(k+1)= Ax(k)+ Bu^\star(k)$ with $u^\star(k)$ as in~\eqref{eq:InfiniteOptimalSequences}.  
	Furthermore, $\mathbb{P}_N$ is in general not convex, but its convex hull is a subset of $\F_N$.
\end{cor}
\begin{proof}
	Let $x(0)\in\mathbb{P}_N$ be arbitrary and let $\a_N$ be the active set for $x(0)$. 
	Since $\a_N$ has the form~\eqref{eq:PersistentActiveSetsHelper1}, 
	it can be partitioned as in~\eqref{eq:SequenceOfPersistentPolytopesHelper1}. 
	According to Corollary~\ref{cor:ActiveSetsShrinkingHorizon} 
	the active set~\eqref{eq:SequenceOfPersistentPolytopesHelper2} defines the optimal solution and polytope for $x^\star(l)$, 
	where $x^\star(l)$ is as in~\eqref{eq:InfiniteOptimalSequences}. 
	Since $\a_{N-l}$ defines a polytope that is persistent from $N$ on according to Corollary~\ref{cor:PersistentPolytopesShrinkingHorizon}
	we have $x^\star(l)\in\mathbb{P}_N$ for all $l\ge 0$, which proves the first claim. 
	The second claim holds, because $\F_N$ is convex~\cite{Bemporad2002-Automatica} and for every convex set $S$ and every subset $S^\prime\subseteq S$ thereof, 
	the convex hull of $S^\prime$ is contained in $S$. 
\end{proof}
All statements made so far apply to the open-loop optimal input sequences and trajectories. Corollary~\ref{cor:MpcOnPersistentPolytopes} also makes a statement about their use on a receding horizon, i.e., about model predictive control (MPC). 
Let $x\rightarrow u^\text{MPC}(x)$, $u^\text{MPC}:\F_N\rightarrow\R^m$ refer to the feedback law that results from applying the first $m$ elements of $\bigu^\star_N:\F_N\rightarrow\R^{mN}$. 
We stress~\eqref{eq:MpcOnPersistentPolytopes} does not hold on $\F_N$ in general (see remark (iii) in the notation section), but multiple problems~\eqref{eq:OCP} have to be solved to determine the MPC input signal sequence in general. 
\begin{cor}\label{cor:MpcOnPersistentPolytopes}
	Let $\mathbb{P}_N$ be defined as in Corollary~\ref{cor:SetOfPersistentPolytopes}. For every $x(0)\in\mathbb{P}_N$, 
	the open-loop optimal input sequence that solves~\eqref{eq:OCP} is equal to the one that results in MPC, i.e.,
	\begin{equation}\label{eq:MpcOnPersistentPolytopes}
	  u^\star(k)= u^\text{MPC}(x^\star(k)),\, k\ge 0,
	\end{equation}
	for $u^\star(k)$ and $x^\star(k)$ as introduced in~\eqref{eq:InfiniteOptimalSequences}. 
	Furthermore, $\mathbb{P}_N$ is positive invariant under MPC. 
\end{cor}
\begin{proof}
  Let $x(0)\in\mathbb{P}_N$ be arbitrary, 
  let $\a_N$ be the active set for $x(0)$
  and let $u^\star(k)$, $x^\star(k)$, $k\ge 0$ be as in~\eqref{eq:InfiniteOptimalSequences}.
  By the same arguments as in the proof of Corollary~\ref{cor:SetOfPersistentPolytopes}, 
  $\a_N$ has the form~\eqref{eq:SequenceOfPersistentPolytopesHelper1}, 
  $\a_{N-l}$ as in~\eqref{eq:SequenceOfPersistentPolytopesHelper2} defines the optimal solution and polytope for $x^\star(l)$ for all $l\in\{0, \dots, N\}$,
  and this polytope is persistent from $N$ on. 
  This implies $u^\star(l)$ is the first optimal input signal for~\eqref{eq:OCP} with horizon $N$ and initial $x^\star(l)$, 
  which yields $u^\text{MPC}(l)= u^\star(l)$ for all $l\in\{0, \dots, N\}$.
  The claim also holds for all $l> N$, because $u^\text{MPC}(x)= K_\infty x$ for all $x\in\T$ 
  and all open-loop optimal input signals are equal to the unconstrained solution on $\T$, 
  i.e.,  $u^\star(k)= K_\infty x^\star(k)$, $x^\star(k+1)= (A+ BK_\infty)x^\star(k)$ for all $k\ge 0$ and all $x(0)\in\T$. 
  Combining~\eqref{eq:MpcOnPersistentPolytopes} and Corollary~\ref{cor:SetOfPersistentPolytopes} 
obviously yields the positive invariance of $\mathbb{P}_N$ under MPC. 
\end{proof} 
\REMARK{All Corollaries after Proposition~\ref{prop:PersistentPolytopes} make statements about persistent polytopes of the form~\eqref{eq:PersistentActiveSetsHelper1}. Remember this form is sufficient for a polytope to be persistent but not necessary.}

\REMARK{Show weaker property than convexity such as connectedness applies, and show that persistent polytopes form onion layers in $N$.
Note that $\mathcal{P}_N$ is not in general equal to its convex hull, but the convex hull is in $\mathcal{F}_N$.}

The statements made so far focused on active sets. 
Some implied statements on the polytopes 
are collected in Remark~\ref{rem:Polytopes} below.
Two types of relations are interesting in particular: Active sets for \textit{different horizons} that define \textit{geometrically equal} polytopes (the persistent polytopes, see part (ii) of Remark~\ref{rem:Polytopes}). 
Secondly, we are interested in sequences of active sets and polytopes that result for the optimally steered system (part (iii) of Remark~\ref{rem:Polytopes}).
The latter essentially extend statements on state trajectories to \textit{trajectories of polytopes}. 
For clarity, let $x^\star(k; x(0), N)$ refer to $x^\star(k)$ for initial condition $x(0)$ and horizon $N$ in Remark~\ref{rem:Polytopes}, 
where $x^\star(k)$ is as defined in~\eqref{eq:InfiniteOptimalSequences}. 
\begin{rem}\label{rem:Polytopes}
(i) Proposition~\ref{prop:ActiveSetExtension} and Corollary~\ref{cor:ActiveSetsShrinkingHorizon}
relate active sets for \textit{shrinking} horizons $N$, $N-1, \dots$, that result from removing stages at the beginning of the horizon.  
The corresponding polytopes are in general not related geometrically, but dynamically. 
Specifically, applying the optimal input signals maps $\P(\a_N)$ into $\P(\a_{N-1})$, 
$\P(\a_{N-1})$ into $\P(\a_{N-2})$ etc., i.e.,
\begin{align}\label{eq:InclusionProperty}
	\{x^\star(l; x(0), N)|x(0)\in\P(\a_{N-l+1})\}\subseteq\P(\a_{N-l})
\end{align}
for all $l\in\{1, \dots, N-1\}$.
The resulting active sets $\a_{N-1}$, $\a_{N-2}, \dots$ and polytopes $\P(\a_{N-1})$, $\P(\a_{N-2}), \dots$ are not defined for horizon $N$, but for the shrinking horizon.

(ii) Lemma~\ref{lem:PersistentActiveSets} and Proposition~\ref{prop:PersistentPolytopes}
discuss a subset of active sets for \textit{growing} horizons $N$, $N+1, \dots$ that result from inserting inactive stages from the end of the horizon. 
The resulting active sets all define the \textit{geometrically same} polytope.
Specifically, 
\begin{align*}
  \P(\a_{N})= \P(\a_{N+l})
\end{align*}
for all $l\ge 0$, where $\a_N$ and $\a_{N+1}$ refer to the sets in~\eqref{eq:PersistentActiveSetsHelper1} and~\eqref{eq:PersistentActiveSetsHelper2}, respectively. 

(iii) Corollary~\ref{cor:PersistentPolytopesShrinkingHorizon} combines the operations from (i) and (ii).
The resulting sequence of polytopes is the optimal sequence from (i) that obeys the inclusion property~\eqref{eq:InclusionProperty}. 
In contrast to (i), however, the new active sets $\tilde{a}_N$ (see~\eqref{eq:SequenceOfPersistentPolytopesHelper3}) are persistent and thus defined for horizon $N$ under the conditions of Corollary~\ref{cor:PersistentPolytopesShrinkingHorizon}. 
\end{rem}

\section{Some computational aspects}\label{sec:ComputationalAspects}
  A simple criterion for the persistency of polytopes is of obvious interest, since it enables us to detect that the infinite-horizon solution has been found for a polytope by solving the simpler finite-horizon problem. 
  If all polytopes are persistent for some $N$, $\mathbb{P}_N= \F_N$ and the solution to the infinite-horizon problem has been found, since $\F_N= \F_{N+l}$ for all $l\ge 0$. 
  A finite $N$ such that $\mathbb{P}_N= \F_N$ does not always exist (see, e.g., \cite{SchulzeDarup2016}), however.  
  In this case $\mathbb{P}_N$ characterizes the largest region for which the solution to the infinite-horizon problem has been found. $\mathbb{P}_N$ is defined by the active sets~\eqref{eq:PersistentActiveSetsHelper1} in a lean fashion. 
  By exploiting its positive invariance, the number of active sets required  to characterize $\mathbb{P}_N$ can be reduced further (see~\eqref{eq:OutmostPersistentAs} below).
  
  The forward constraint order~\eqref{eq:ForwardConstraintOrder} arises in backward dynamic programming and therefore is an obvious choice. 
  The backward order, i.e., the order that results from reversing the sequence of lines in~\eqref{eq:ForwardConstraintOrder}, may be more useful in computations. 
  For example, the bit tuples~\eqref{eq:ActiveSetExtensionSampleAs} in Example~\ref{ex:ActiveSetExtension} 
  correspond to the index sets $\{1\}$, $\{7\}$, $\{1, 7\}$ and $\{2, 7\}$
  according to rule~\eqref{eq:BitTupels}. 
  While the relation of the active sets is immediately evident from the bit tuples~\eqref{eq:ActiveSetExtensionSampleAs}, this is not the case for the set notation, since the introduction of the additional stage for $N= 2$ shifts all indices by $q_\X+ q_\U= 6$.
  If the backward constraint order is used, 
  the bit tuples and index sets that correspond to~\eqref{eq:ActiveSetExtensionSampleAs} read 
  \begin{equation}\label{eq:BackwardActiveSetExtensionAs}
    \begin{tabular}{ll}
	  $0000.100000$ & $\{5\}$ 
	  \\
	  $0000.100000.000000$ & $\{5\}$ 
	  \\
	  $0000.100000.100000$ & $\{5, 11\}$ 
	  \\
	  $0000.100000.010000$ & $\{5, 12\}$ 
    \end{tabular}
  \end{equation} 
  and their relation is obvious in both notations. 
  The statements of the paper carry over to the backward order in an obvious fashion. 
  For example, the $q_\T$ trailing zeros in all statements in Section~\ref{sec:Persistency} and~\eqref{eq:ActiveSetExtensionSampleAs}
  become $q_\T$ leading zeros in the corresponding statements and~\eqref{eq:BackwardActiveSetExtensionAs} for the backward order.
  Note that the order of constraints within each stage must be fixed but is irrelevant.

  Corollary~\ref{cor:PersistentPolytopesShrinkingHorizon} suggests to determine and store the outmost persistent polytopes 
  and to determine the remaining persistent polytopes with simple bit shifting operations. 
  More precisely, assume all active sets of the form
  \begin{align}\label{eq:OutmostPersistentAs}
    \a_N= \a_{N,0} \cdots \underbrace{\a_{N,N-1}}_{\ne 0\cdots 0}\underbrace{0\cdots 0}_{q_\T}.
  \end{align}
  where the $\a_{N,i}$, $i=0, \dots, N-2$ are arbitrary tuples (possibly $0\cdots 0$) of length $q_\X+ q_\U$.
  According to Corollary~\ref{cor:PersistentPolytopesShrinkingHorizon} the active sets
    \begin{align}\label{eq:PersistentOffspring}
    \a_{N, l} \cdots \a_{N,N-1}
    \underbrace{0\dots0}_{l\cdot(q_\X+ q_\U)}
    \underbrace{0\cdots 0}_{q_\T},
    \quad
    l= 1, \dots, N-1
  \end{align}   
  can be generated for each~\eqref{eq:OutmostPersistentAs} and define persistent polytopes. 
  They define a trajectory of polytopes for the initial polytope~\eqref{eq:OutmostPersistentAs} according to part (iii) of Remark~\ref{rem:Polytopes} that leads to $\T$. 
  Combining $\T$ and all polytopes defined by~\eqref{eq:OutmostPersistentAs} and~\eqref{eq:PersistentOffspring} yields the subset of persistent polytopes $\mathbb{P}_N\subseteq\mathcal{F}_N$. 
  Since the active sets~\eqref{eq:PersistentOffspring} are constructed by simple bit shifting operations, this is a computationally attractive approach to constructing $\mathbb{P}_N$ from the sets~\eqref{eq:OutmostPersistentAs}. 
  The active sets~\eqref{eq:OutmostPersistentAs} can be determined by solving linear programs~\cite{Gupta2011}. 
  
  We claim without giving details that the proposed approach results in a particular depth-first analysis of the active set tree first proposed in~\cite{Gupta2011}. 
  It is an obvious question whether the computational effort of the method proposed in~\cite{Gupta2011} and refined versions thereof~\cite{Feller2013,Herceg2015,Oberdieck2017} can be reduced with the results presented here. The active set tree grows exponentially in $q$, thus it grows exponentially in $N$, $m$ and $n$ in typical cases (for example, $q= 2(N+1)n+ 2Nm$ for bound constraints), and consequently any method for discarding candidate active sets is of great interest. 
  Such an analysis would also have to include a comparison to other well established methods and implementations~\cite{MatlabMpcToolbox,Herceg2013} for solving~\eqref{eq:OCP}. 
  The focus of the present paper is not on computational methods, however, and a comparison is beyond the present paper.

\section{Conclusions and future work}\label{sec:Conclusions}
We uncovered a certain structure of the set of active sets that define the solution to the constrained linear quadratic regulator. 
While the set of active sets and the set of affine pieces and polytopes are equally useful in that both define the solution of the same optimal control problem, 
the structure of the former revealed here is not immediately evident in the latter. 
We therefore claim the structure of the set of active sets is interesting and important per se. This is corroborated by the fact that very simple operations (such as deleting bit tuples that represent a stage, or inserting zeroes for another stage) suffice to generate the active sets that define the persistent part of the geometric solution, i.e., the persistent affine pieces and polytopes. 
More practically, the structure of the set of active sets is useful, for example, 
for an analysis and a lean characterization of those parts of the solution that are independent of the horizon $N$,  
which after all is a nuisance parameter in optimal control problem (where "independence" is understood as in Definition~\ref{def:PersistentActiveSets}). 


Future work will address the extension to nonlinear optimal control problems
and investigate the computational aspects summarized in Section~\ref{sec:ComputationalAspects}.

\section*{Appendix}

\appendix
Lemma~\ref{lem:PositiveInvarianceAndInterior} is used in the proof of Lemma~\ref{lem:PersistentActiveSets} for $\mathcal{S}=\T$ and $\bar{A}= A+BK_\infty$. 
\begin{lem}\label{lem:PositiveInvarianceAndInterior}
	Let  $\mathcal{S}\subset\R^n$ be a compact set. If $\mathcal{S}$ is positive invariant for the $n$-dimensional system $x(k+1)= \bar{A}x(k)$, then
	the following statements hold:
	\begin{enumerate}
		\item For any $\lambda\in (0, 1)$,  the set $\lambda\mathcal{S}=\{\lambda \xi | \xi\in\mathcal{S}\}$ is positive invariant.  
		\item The interior of $\mathcal{S}$ is positive invariant.
	\end{enumerate}
\end{lem}
\REMARK{The proof violates the convention that $x$ should not be used:}
\begin{proof}
Let $\lambda\in(0, 1)$ be arbitrary. Let $\zeta\in\lambda\mathcal{S}$ be arbitrary, then there 
exists an $\xi\in\mathcal{S}$ such that $\zeta= \lambda \xi$ by definition of $\lambda\mathcal{S}$.  
Since $\mathcal{S}$ is positive invariant, $\xi\in\mathcal{S}$ implies $\bar{A}\xi\in\mathcal{S}$, which implies $\lambda\bar{A}\xi\in\lambda\mathcal{S}$. 
Combining this with $\lambda\bar{A}\xi= \bar{A}\lambda \xi= \bar{A}\zeta$ yields $\bar{A}\zeta\in\lambda\mathcal{S}$. 
Since $\lambda\in (0, 1)$ and $\zeta\in\lambda\mathcal{S}$ were arbitrary, the first claim holds.
The second claim follows, since, for any $\xi\in\text{int}\,\mathcal{S}$ there exists a $\lambda\in(0, 1)$ such that $\xi\in\lambda\mathcal{S}$ and $\lambda\mathcal{S}\subset\text{int}\,\mathcal{S}$. 
\end{proof}
\bibliographystyle{plain}
\bibliography{bibliography}
\end{document}